\documentclass[11pt, a4paper]{article}
\usepackage{typearea}
\typearea{12}

\usepackage{amsmath, amssymb, enumerate, amsthm, accents}
\usepackage{amsmath, amssymb, url, comment}
\usepackage[numbers]{natbib}
\usepackage{tikz}
\usepackage{pgfplotstable}
\usetikzlibrary{arrows,positioning,plotmarks,external,patterns,angles,decorations.pathmorphing,backgrounds,fit,shapes,graphs,calc,spy}
\pgfplotsset{compat=1.14}

\def\dint{\int\!\!\!\int}

\theoremstyle{plain}
\newtheorem{theorem}{Theorem}[section]
\newtheorem{lemma}[theorem]{Lemma}

\newtheorem{proposition}[theorem]{Proposition}

\theoremstyle{remark}

\begin{document}
\title{Singular value shrinkage priors for Bayesian prediction}
\author{Takeru Matsuda\thanks{Department of Mathematical Informatics, Graduate School of Information Science and Technology, University of Tokyo, 7-3-1 Hongo, Bunkyo-ku, Tokyo 113-0033, Japan} \and Fumiyasu Komaki\footnotemark[1]}
	\date{}
	
	\maketitle

\begin{abstract}
	We develop singular value shrinkage priors for the mean matrix parameters in the matrix-variate normal model with known covariance matrices.
	Our priors are superharmonic and put more weight on matrices with smaller singular values. 
	They are a natural generalization of the Stein prior.
	Bayes estimators and Bayesian predictive densities based on our priors are minimax and dominate those based on the uniform prior in finite samples.
	%The risk reduction is large when the true value of the parameter has small singular values.
	In particular, our priors work well when the true value of the parameter has low rank.
\end{abstract}

\footnote[0]{This is a pre-copyedited, author-produced version of an article accepted for publication in \textit{Biometrika} following peer review. The version of record (T. Matsuda and F. Komaki. Singular value shrinkage priors for Bayesian prediction. \textit{Biometrika} \textbf{102}, 843--854, 2015.) is available online at \url{https://doi.org/10.1093/biomet/asv036}.}

\section{Introduction}
Suppose that we have a matrix observation $Y \sim {\rm N}_{n,m} (M, C, \Sigma)$ where $C$ and $\Sigma$ are known.
Here we use the notation of \cite{Dawid} for matrix-variate normal distributions:
$Y \sim {\rm N}_{n,m} (M, C, \Sigma)$
indicates that $Y (n \times m)$ has a probability density
\begin{align*}
p (Y)=& \frac{1}{(2 \pi)^{n m/2} (\det C)^{m/2} (\det \Sigma)^{n/2}} {\rm etr} \left\{ -\frac{1}{2} \Sigma^{-1} (Y-M)^{\top} C^{-1} (Y-M) \right\}
\end{align*}
with respect to the Lebesgue measure on $\mathbb{R}^{n \times m}$, where $\mathrm{etr} (A) = \exp (\mathrm{tr } A)$, $M(n \times m)$ is the mean matrix, 
$C (n \times n) \succ 0$ is the covariance matrix for rows, and $\Sigma (m \times m) \succ 0$ is the covariance matrix for columns.
Here, we indicate the size of a matrix $A \in \mathbb{R}^{n \times m}$ by writing $A (n \times m)$.
The vectorization ${\rm vec} (Y)$ of $Y$ satisfies
\begin{equation}
{\rm vec} (Y) \sim {\rm N}_{m n} \{ {\rm vec} (M), \Sigma \otimes C \}, \label{vectorization}
\end{equation}
where $\otimes$ denotes the Kronecker product \citep{Gupta}.
Here, the vectorization of $A (n \times m)$ is the $m n \times 1$ vector defined by
${\rm vec} (A) = (a_{11}, \ldots, a_{n1}, a_{12}, \ldots, a_{n2}, \ldots, a_{1m}, \ldots, a_{nm})^{\top}$,
and the Kronecker product $A \otimes B$ of two matrices $A (m \times n) = (a_{ij})$ and $B (p \times q) = (b_{ij})$ is the $mp \times nq$ matrix 
\begin{equation*}
A \otimes B = \begin{pmatrix} a_{1 1} B & a_{1 2} B & \cdots & a_{1 n} B \\
a_{2 1} B & a_{2 2} B & \cdots & a_{2 n} B \\
\vdots & \vdots & \ddots & \vdots \\
a_{m 1} B & a_{m 2} B & \cdots & a_{m n} B \\
\end{pmatrix}.
\end{equation*}

We consider the prediction of $\widetilde{Y} \sim {\rm N}_{n,m} (M, \widetilde{C}, \widetilde{\Sigma})$ by a predictive density $\hat{p} (\widetilde{Y} \mid Y)$, where $\widetilde{C}$ and $\widetilde{\Sigma}$ are known.
We evaluate predictive densities by the Kullback--Leibler divergence
\begin{equation*}
D \{\widetilde{p}(\cdot \mid M), \hat{p} (\cdot \mid Y)\}
= \int \widetilde{p}(\widetilde{Y} \mid M) \log \frac{\widetilde{p}(\widetilde{Y} \mid M)}{\hat{p} (\widetilde{Y} \mid Y)} {\rm d} \widetilde{Y}
\end{equation*}
as a loss function. The Kullback--Leibler risk function of a predictive density $\hat{p} (\widetilde{Y} \mid Y)$ is
\begin{align*}
R_{{\rm KL}} (M, \hat{p}) &=  E \left[ D \left\{ \widetilde{p}(\cdot \mid M), \hat{p} (\cdot \mid Y) \right\} \right] = \dint p(Y \mid M) \widetilde{p}(\widetilde{Y} \mid M) \log \frac{\widetilde{p}(\widetilde{Y} \mid M)}{\hat{p} (\widetilde{Y} \mid Y)} {\rm d} Y {\rm d} \widetilde{Y}.
\end{align*}
We consider Bayesian predictive densities with a prior $\pi(M)$:
\begin{align*}
\hat{p}_{\pi} (\widetilde{Y} \mid Y) & = \int \widetilde{p} (\widetilde{Y} \mid M) \pi (M \mid Y) {\rm d} M
= \frac{\int \widetilde{p} (\widetilde{Y} \mid M) p (Y \mid M) \pi (M) {\rm d} M}{\int \widetilde{p} (\widetilde{Y} \mid M) \pi (M) {\rm d} M}.
\end{align*}

Multivariate linear regression can be formulated as the prediction of the matrix-variate normal model.
Suppose that we predict $\widetilde{Y} \sim {\rm N}_{m,q} (\widetilde{X} B, \sigma^2 I_m, I_q)$ based on $Y \sim {\rm N}_{n,q} (X B, \sigma^2 I_n, I_q)$,
where $X (n \times p)$ and $\widetilde{X} (m \times p)$ are explanatory variables,
$Y (n \times q)$ and $\widetilde{Y} (m \times q)$ are response variables, 
$B (p \times q)$ is a regression coefficient matrix and $\sigma^2$ is a known variance.
%where $X (n \times p)$ and $\widetilde{X} (m \times p)$ are matrices of explanatory variables,
%$Y (n \times q)$ and $\widetilde{Y} (m \times q)$ are matrices of response variables, 
%$B (p \times q)$ is a regression coefficient matrix and $\sigma^2$ is a known variance.
Since $Y_1 = (X^{\top} X)^{-1} X^{\top} Y$ and $\widetilde{Y}_1 = (\widetilde{X}^{\top} \widetilde{X})^{-1} \widetilde{X}^{\top} \widetilde{Y}$ are sufficient for $B$,
the problem reduces to the prediction of $\widetilde{Y}_1 \sim {\rm N}_{p,q} \{B, \sigma^2 (\widetilde{X}^{\top} \widetilde{X})^{-1}, I_q \}$
based on $Y_1 \sim {\rm N}_{p,q} \{ B, \sigma^2 (X^{\top} X)^{-1}, I_q \}$.

The prediction of $\widetilde{y} \sim {\rm N}_n (\mu, \widetilde{\Sigma})$ based on $y \sim {\rm N}_n (\mu, \Sigma)$,
which corresponds to $m=1$, has been studied by several authors when $n \geq 3$.
When $\widetilde{\Sigma}$ is proportional to $\Sigma$,
\cite{Komaki01} gave the analytical form of Bayesian predictive densities based on the Stein prior
$\pi_{{\rm S}} (\mu) = \| \mu \|^{-(n-2)}$
and proved that Bayesian predictive densities based on this prior dominate those based on the Jeffreys prior under the Kullback--Leibler risk.
Since the Stein prior puts more weight near the origin, the risk reduction is large when the true value of $\mu$ is near the origin.
\cite{George06} generalized this result and 
proved that Bayesian predictive densities based on superharmonic priors dominate those based on the Jeffreys prior. % under the Kullback--Leibler risk.
\cite{Kobayashi} and \cite{George08} 
considered cases where $\widetilde{\Sigma}$ is not necessarily proportional to $\Sigma$.
Bayesian predictive densities based on superharmonic priors also dominate 
%those based on the uniform prior under the Kullback--Leibler risk in this general situation. 
those based on the uniform prior in this general situation. 
They applied their results to linear regression.

In the following, we assume that $n-m \geq 2$.
Since matrix-variate normal distributions are special cases of vector-variate normal distributions as in \eqref{vectorization},
above results also apply to the former distributions.
In this paper, we propose superharmonic priors, which shrink the singular values of $M$
and are a natural generalization of the Stein prior.
Bayes estimators and Bayesian predictive densities based on our priors are minimax and dominate those based on the uniform prior.
The risk reduction is larger when the true $M$ has smaller singular values,
so our priors work particularly well when the true $M$ has low rank.
%since low rank matrices have sparse singular values. 
In multivariate linear regression, we can reasonably expect $M$ to have low rank \citep{Reinsel}.
Previously proposed superharmonic priors mainly shrink the posteriors
to simple subsets such as a point or a linear subspace of the parameter space.
In contrast, our priors shrink the posteriors to the sets of low-rank matrices.

Our priors have several advantages over previously proposed priors for the matrix-variate normal model.
\cite{Tsukuma08} proposed hierarchical priors that are a natural generalization of Strawderman's prior 
and proved admissibility and minimaxity of the Bayes estimators based on them.
However, it is unknown whether Tsukuma's priors perform well in prediction.
Our priors provide good prediction as well as estimation.
Also, only simple Monte Carlo sampling from the normal distribution is sufficient to calculate Bayes estimators and Bayesian predictive densities based on our priors, 
whereas Markov chain Monte Carlo methods are required to obtain Bayes estimators based on Tsukuma's priors.

\section{Singular value shrinkage priors}
\subsection{Singular value shrinkage estimators}
Singular value shrinkage was utilized by \cite{Efron} and \cite{Stein74} for estimation.
Here we review their work.
In this subsection, we assume that $Y \sim {\rm N}_{n,m} (M, I_n, I_m)$, where $I_k$ is the $k$-dimensional identity matrix.
We consider the estimation of $M$.

Let
$Y = U \Lambda V^{\top}$,
$U \in O(n)$, $V \in O(m)$, $\Lambda = \left\{ {\rm diag} (\sigma_1, \ldots, \sigma_m) \; O_{m,n-m} \right\}^{\top}$
be the singular value decomposition of a matrix $Y$, where $O_{m,n-m}$ is the zero matrix of size $m \times (n-m)$, and
$\sigma_1 \geq \cdots \geq \sigma_m \geq 0$ are the singular values of $Y$.
Similarly, let $\hat{M} = \hat{U} \hat{\Lambda} \hat{V}^{\top}$,
$\hat{U} \in O(n)$, $\hat{V} \in O(m)$, $\hat{\Lambda} = \left\{ {\rm diag} (\hat{\sigma}_1, \cdots, \hat{\sigma}_m) \; O_{m,n-m} \right \}^{\top}$
be the singular value decomposition of an estimator $\hat{M}$ of $M$, 
where $\hat{\sigma}_1 \geq \cdots \geq \hat{\sigma}_m \geq 0$ are the singular values of $\hat{M}$.

\cite{Efron} proposed
$\hat{M}_{{\rm EM}} = Y \left\{ I_m-(n-m-1)S^{-1} \right\}$ %\label{EMestimator}
as an empirical Bayes estimator, where $S=Y^{\top} Y$.
They proved that $\hat{M}_{{\rm EM}}$ is minimax and dominates the maximum likelihood estimator under the Frobenius loss
$l(M,\hat{M}) = \| \hat{M}-M \|_{{\rm F}}^2 = \sum_{i,j} (\hat{M}_{i j} - M_{i j})^2$.
\cite{Stein74} noticed that $\hat{M}_{{\rm EM}}$ can be represented in the singular value decomposition form as
\begin{equation*}
\hat{\sigma}_i = \left( 1 - \frac{n-m-1}{\sigma_i^2} \right) \sigma_i \quad (i=1, \ldots, m), \quad 
%\end{equation*}
%\begin{equation*}
\hat{U} = U, \quad \hat{V} = V.
\end{equation*}
Therefore, $\hat{M}_{{\rm EM}}$ shrinks the singular values of the observation $Y$.
When $m=1$, $\hat{M}_{{\rm EM}}$ coincides with the James--Stein estimator.
We note that $\hat{M}_{{\rm EM}}$ is not a Bayes estimator.

In this study, we develop superharmonic priors shrinking the singular values.
The Bayes estimators based on our priors have similar properties to $\hat{M}_{{\rm EM}}$.
This is an extension of the relationship between the James--Stein estimator and the Stein prior.

\subsection{Definition and superharmonicity}
We consider the prior 
\begin{equation}
\pi_{{\rm SVS}} (M) = \det (M^{\top} M)^{-(n-m-1)/2}, \label{SV_shrinkage_prior}
\end{equation}
with $n-m \geq 2$.
From the relation $ \det (M^{\top} M) = \prod_{i=1}^m \sigma_i(M)^2$,
where $\sigma_i(M)$ denotes the $i$-th singular value of $M$, we obtain
\begin{equation*}
\pi_{{\rm SVS}} (M) = \prod_{i=1}^m \sigma_i(M)^{-(n-m-1)}.
\end{equation*}
Therefore, this prior puts more weight on matrices with smaller singular values.
When $m=1$, $\pi_{{\rm SVS}}$ coincides with the Stein prior.
We call $\pi_{{\rm SVS}}$ the singular value shrinkage prior below.

We provide a proof of superharmonicity of  $\pi_{{\rm SVS}}$.
An extended real-valued function $f: \mathbb{R}^d \to \mathbb{R} \cup \{ \infty \}$ is said to be superharmonic if it satisfies the following properties \citep[p.~70]{Helms}:

\begin{enumerate}
	\item $f$ is lower semicontinuous;
	
	\item $f \not\equiv \infty$;
	
	\item $L(f: x, \delta)= \int_{S_{x,\delta}} f(z) {\rm d} s(z) / (\Omega_d \delta^{d-1}) \leq f(x)$ for every $x \in \mathbb{R}^d$ and $\delta > 0$,
	where $\Omega_d$ is the surface area of the unit sphere in $\mathbb{R}^d$ and $S_{x,\delta}$ is the sphere with center $x$ and radius $\delta$.
\end{enumerate}

If $f$ is a $C^2$ function, then $f$ is superharmonic if and only if $\Delta f (x) = \sum_{i=1}^d (\partial^2 / \partial x_i^2) f(x) \leq 0$
holds for every $x$ from Lemma 3.3.4 of \cite{Helms}.
We define a function $f: \mathbb{R}^{n \times m} \to \mathbb{R}$ to be superharmonic 
when $f \circ {\rm vec}: \mathbb{R}^{m n} \to \mathbb{R}$ is superharmonic.

\begin{theorem} \label{th_piSVS_superharmonic}
	The prior density $\pi_{{\rm SVS}}$ is superharmonic.
\end{theorem}

The proof is deferred to the Appendix.
From this proof, we obtain the following Theorem.

\begin{theorem}
	If $M$ has full rank, then the prior $\pi_{{\rm SVS}}$ satisfies $\Delta \pi_{{\rm SVS}} (M) = 0$.
\end{theorem}

%\begin{proof}
%Consider $\varepsilon \to 0$ in $\eqref{A+B}$.
%\end{proof}

Therefore, the superharmonicity of $\pi_{{\rm SVS}}$ is strongly concentrated 
in the same way as the Laplacian of the Stein prior becomes a Dirac delta function.

Another interesting point about this prior is that $\pi_{{\rm SVS}}$ is superharmonic column-wise:
$\pi_{{\rm SVS}}$ is superharmonic as a function of the $i$-th column of $M$ ($i=1,\ldots,m$).

\subsection{Minimaxity of Bayes estimators and Bayesian predictive densities based on $\pi_{{\rm SVS}}$}
We prove that the Bayes estimators and the Bayesian predictive densities based on $\pi_{{\rm SVS}}$ are minimax and dominate those based on the Jeffreys prior.

When $Y \sim {\rm N}_{n,m} (M,C,\Sigma)$, we define the marginal distribution of $Y$ with prior $\pi(M)$ by
\begin{equation*}
m_{\pi} (Y; C,\Sigma) = \int p(Y \mid M,C,\Sigma) \pi (M) {\rm d} M.
\end{equation*}
When $\pi = \pi_{{\rm SVS}}$, we denote the marginal distribution by $m_{{\rm SVS}}$.

\begin{lemma} \label{th_m_superharmonic}
	If $m_{\pi} (Y; C,\Sigma) < \infty$ for every $Y$ and $\pi$ is superharmonic, then $m_{\pi} (Y; C,\Sigma)$ is superharmonic.
\end{lemma}

The proof is deferred to the Appendix.

In terms of estimation of $M$ from $Y$, the following result \citep{Stein74} is known.

\begin{lemma} \label{th_stein}
	If $m_{\pi}$ satisfies $\Delta m_{\pi} (Y; C,\Sigma) \leq 0$, then Bayes estimator $\hat{M}^{\pi}$ with prior $\pi$ is minimax under the Frobenius loss.
	Furthermore, Bayes estimator with prior $\pi$ dominates the maximum likelihood estimator unless $\pi$ is the uniform prior.
\end{lemma}

The lemma above is obtained from the expression
\begin{align}
E_{M} & \left( \| \hat{M}^{{\rm MLE}} - M \|_{{\rm F}}^2 \right)
- E_{M} \left( \| \hat{M}^{\pi} - M \|_{{\rm F}}^2 \right)
= E_{M} \left\{ \| \nabla \log m_{\pi} (Y) \|^2 - 2 \frac{\Delta m_{\pi} (Y)}{m_{\pi} (Y)} \right\} \label{est_risk_diff}
\end{align}
of the risk difference between the maximum likelihood estimator and the Bayes estimator.% with prior $\pi$.

In terms of prediction, we obtain the following result when $\widetilde{\Sigma} \otimes \widetilde{C}$ is proportional to $\Sigma \otimes C$.
This corresponds to the setting of \cite{Komaki01} and \cite{George06}.
The general covariance case is considered in Section 2.5.
We assume $C = v_1 I_n, \widetilde{C} = v_2 I_n, \Sigma = \widetilde{\Sigma} = I_m$ without loss of generality.
Let $v_0 = (v_1 v_2)/(v_1+v_2) < v_1$.
We write the Bayesian predictive density based on the uniform prior $\pi_{{\rm I}}(M) \equiv 1$ by $\hat{p}_{{\rm I}} (\widetilde{Y} \mid Y)$.

\begin{lemma} \label{lemma_helms} \citep[Lemma 3.4.4]{Helms}
	If $f$ is superharmonic, then $L(f : x, \delta)=\int_{S_{x,\delta}} f(z) {\rm d} s(z) / (\Omega_d \delta^{d-1})$ is a decreasing function of $\delta$.
\end{lemma}

\begin{proposition} \label{th_george}
	If $\pi$ is superharmonic, then $\hat{p}_{\pi} (\widetilde{Y} \mid Y)$ is minimax under the Kullback--Leibler risk $R_{{\rm KL}}$.
	Furthermore, $\hat{p}_{\pi} (\widetilde{Y} \mid Y)$ dominates $\hat{p}_{{\rm I}} (\widetilde{Y} \mid Y)$ unless $\pi$ is the uniform prior.
\end{proposition}

The proof is deferred to the Appendix.
Proposition $\ref{th_george}$ does not require an assumption concerning twice differentiablility of $\pi$ and is
a slight generalization of the results in \cite{George06}.

From Theorem $\ref{th_piSVS_superharmonic}$ and Lemma $\ref{th_m_superharmonic}$, $m_{{\rm SVS}}$ is superharmonic.
Also, from $\eqref{marginal}$, $m_{{\rm SVS}}$ is a $C^2$ function.
Therefore, $\pi_{{\rm SVS}}$ satisfies the conditions of Lemma $\ref{th_stein}$ and Proposition $\ref{th_george}$.
By combining these results, we obtain the following.

\begin{theorem}
	The Bayes estimator based on the prior $\pi_{{\rm SVS}}$
	is minimax and dominates the maximum likelihood estimator under the Frobenius risk.
	The Bayesian predictive density $\hat{p}_{{\rm SVS}} (\widetilde{Y} \mid Y)$ based on the prior $\pi_{{\rm SVS}}$
	is minimax and dominates $\hat{p}_{{\rm I}} (\widetilde{Y} \mid Y)$ under the Kullback--Leibler risk $R_{{\rm KL}}$.
\end{theorem}

Since $\pi_{{\rm SVS}}$ shrinks each singular value separately, 
the risk reduction of $\pi_{{\rm SVS}}$ is larger when the true value of $M$ has smaller singular values.
A remarkable point is that $\pi_{{\rm SVS}}$ works well even when only some of the singular values are small.
In particular, $\pi_{{\rm SVS}}$ works well when $M$ has low rank.
We confirm these facts by numerical experiments in Section 3.

\subsection{Bayesian predictive density based on $\pi_{{\rm SVS}}$}
We provide the analytical form of the Bayesian predictive densities based on $\pi_{{\rm SVS}}$.
Here, we consider the case where $ \widetilde{\Sigma} \otimes \widetilde{C}$ is proportional to $\Sigma \otimes C$
and assume $C = \widetilde{C} = I_n$ and $\Sigma = \widetilde{\Sigma} = I_m$ without loss of generality.

\begin{theorem} \label{th_predictive_density}
	The Bayesian predictive density based on the prior $\pi_{{\rm SVS}}$ is
	\begin{align}
	\hat{p}_{{\rm SVS}} (\widetilde{Y} \mid Y) & =
	(2 \pi)^{-m n / 2} {\rm etr} \left\{ -\frac{1}{4} (\widetilde{Y}-Y)^{\top} 
	(\widetilde{Y}-Y) -\frac{1}{4} Z^{\top} Z + \frac{1}{2} Y^{\top} Y \right\} \nonumber \\
	&  \quad \times \frac{2^{-m (m+1)/2} {}_1 F_1 
		\left( \frac{m+1}{2}; \frac{n}{2}; \frac{1}{4} Z^{\top} Z \right)}
	{{}_1 F_1 \left( \frac{m+1}{2}; \frac{n}{2}; \frac{1}{2} Y^{\top} Y \right)}. \label{SVshr_density}
	\end{align}
	Here, $Z=Y+\widetilde{Y}$ and ${}_p F_q$ is the hypergeometric function of matrix argument \citep[p.~34]{Gupta}
	defined by
	\begin{equation*}
	{}_p F_q (a_1, \ldots, a_p; b_1, \ldots, b_q; S) = \sum_{k=0}^{\infty} \sum_{\kappa \vdash k} \frac{(a_1)_{\kappa} \cdots (a_p)_{\kappa}}{(b_1)_{\kappa} \cdots (b_q)_{\kappa}} \frac{C_{\kappa} (S)}{k!},
	\end{equation*}
	where $a_1, \ldots, a_p, b_1, \ldots, b_q$ are arbitrary complex numbers, $S$ is a $p \times p$ complex symmetric matrix, $\sum_{\kappa \vdash k}$ denotes summation over all partitions $\kappa$ of $k$, $(a)_{\kappa}$ is the generalized Pochhammer symbol, and $C_{\kappa} (S)$ denotes the zonal polynomial.
\end{theorem}

The proof is deferred to the Appendix.

\subsection{General covariance case}
%Thus far, we assumed that the covariance matrices of observation data $C$ and $\Sigma$ and of future data 
%$\widetilde{C}$ and $\widetilde{\Sigma}$ are the same.
Thus far, we assumed that $\widetilde{\Sigma} \otimes \widetilde{C}$ is proportional to $\Sigma \otimes C$.
However, this assumption does not hold in regression problems \citep{Kobayashi,George08}.
%We extend the results for this general situation to matrix-variate case.
Here, we consider the general covariance case.
We define $\Sigma_1 = \left\{ (\Sigma \otimes C)^{-1}+(\widetilde{\Sigma} \otimes \widetilde{C})^{-1} \right\}^{-1},$
$\Sigma_2 = \Sigma \otimes C$ and write the diagonalization of $\Sigma_1^{1/2} \Sigma_2^{-1} \Sigma_1^{1/2}$ as
$\Sigma_1^{1/2} \Sigma_2^{-1} \Sigma_1^{1/2} = U^{\top} \Lambda U$,
where $U$ is an orthogonal matrix and $\Lambda$ is a diagonal matrix.
Let $A^{*} = \Sigma_1^{1/2} U^{\top} (\Lambda^{-1} - I)^{1/2}$.
From Theorem 3.2 of \cite{Kobayashi}, we obtain the following.

\begin{theorem} \label{th_koba}
	If $\pi \left[ {\rm vec}^{-1} \left\{ A^{*} {\rm vec} (M) \right\} \right]$ is superharmonic as a function of $M$, 
	$\hat{p}_{\pi} (\widetilde{Y} \mid Y)$ dominates $\hat{p}_{{\rm I}} (\widetilde{Y} \mid Y)$ under $R_{{\rm KL}}$.
\end{theorem}

We can construct a prior that satisfies the condition of Theorem $\ref{th_koba}$ by using $\pi_{{\rm SVS}}$:
\begin{equation}
\pi(M) = \pi_{{\rm SVS}} \left[ {\rm vec}^{-1} \left\{ (A^{*})^{-1} {\rm vec} (M) \right\} \right]. \label{prior_koba}
\end{equation}
The analytical form of the Bayesian predictive densities based on the prior $\eqref{prior_koba}$ is unknown.
We conjecture that some extended generalized Laguerre polynomial of matrix argument is necessary to obtain this form.
%In practice, numerical evaluation of the Bayesian predictive density is available with Monte Carlo sampling.

In multivariate linear regression, the regression coefficient matrix $B$ often has low rank \citep{Reinsel}.
The prior $\eqref{prior_koba}$ works particularly well in such situations.
We confirm this by numerical experiments in Section 3.

\section{Numerical results}
In this section, we show numerical results on Bayes estimation and Bayesian prediction with singular value shrinkage priors.
We compare the singular value shrinkage prior to the Jeffreys and the Stein priors.
Here, the Jeffreys prior coincides with the uniform prior and the Stein prior is $\pi_{{\rm S}} (M) = \| M \|_{\mathrm{F}}^{-(n m -2)}$, 
where $\| M \|_{\mathrm{F}}$ denotes the Frobenius norm of $M$. We note that$\| M \|_{\mathrm{F}}^2 = \sum_{i=1}^m \sigma_i^2$, 
where $\sigma_i$ is the $i$th singular value of $M$. 
All the computations took less than 5 seconds on a laptop computer.

%We used two methods to calculate the hypergeometric function of matrix argument in $\eqref{SVshr_density}$.
%First is an algorithm by \cite{Koev} that exploits the combinatorial properties of the Jack function.
%Second is the holonomic gradient method of \cite{Hashiguchi}, wherein we numerically solve a PDE that ${}_1 F_1$ satisfies.
%While the first method does not wok well for matrix arguments with large singular values, the second method can in principle be used for any matrix arguments.
The hypergeometric function of matrix argument ${}_1 F_1$ is calculated as follows.
From Theorem 3.5.6 in \cite{Gupta},
\begin{align*}
{}_1 F_1 \left( \frac{m+1}{2}; \frac{n}{2}; \frac{1}{2} Y^{\top} Y \right)
&= 
\frac{2^{\frac{m(n-m-1)}{2}} \Gamma_m \left( \frac{n}{2} \right)}
{\Gamma_m \left( \frac{m+1}{2} \right)} {\rm etr} \left( \frac{1}{2} Y^{\top} Y \right)
E \left\{ (\det S)^{-(n-m-1)/2} \right\},
\end{align*}
where $S \sim W_m (n, I_m, Y^{\top} Y)$, and we take
\begin{align*}
E \left\{ (\det S)^{-(n-m-1)/2} \right\} & \approx \frac{1}{N} \sum_{i=1}^N (\det M_i^{\top} M_i)^{-(n-m-1)/2},
\end{align*}
where $M_1, \ldots, M_N$ are independent samples from ${\rm N}_{n,m} (Y, I_n, I_m)$, with $N = 10^4$.

First, we investigate Bayes estimation. 
The estimation of $M$ from $Y \sim {\rm N}_{n,m} (M, I_n, I_m)$ is considered.
From \cite{Stein74}, a Bayes estimator with prior $\pi(M)$ is
\begin{equation*}
\left\{ \widehat{M}^{\pi} (Y) \right\}_{i j} = Y_{i j} + \frac{\partial}{\partial Y_{i j}} \log m_{\pi} (Y),
\end{equation*}
where $m_{\pi} (Y)$ is the marginal distribution of $Y$ with prior $\pi(M)$.
When $\pi = \pi_{{\rm SVS}}$, by $\eqref{marginal}$,
\begin{equation*}
\frac{\partial}{\partial Y_{i j}} \log m_{{\rm SVS}} (Y) = -Y_{i j} + \frac{\partial}{\partial Y_{i j}} \left\{ \log {}_1 F_1 \left( \frac{m+1}{2}; \frac{n}{2}; \frac{1}{2} Y^{\top} Y \right) \right\}.
\end{equation*}
We approximated the partial differentiation of $\log {}_1 F_1$ by finite differencing.

We compare the risk functions of the Bayes estimator with the prior $\pi_{{\rm SVS}}$
%to the Bayes estimator with the Jeffreys prior and the Bayes estimator with the Stein prior.
to the Jeffreys prior and the Stein prior.
The second estimator coincides with the maximum likelihood estimator.
We sampled $Y$ $10^4$ times and approximated the risk by the sample mean of the Frobenius loss.

Figure $\ref{sv_shr_fig}$ (a) shows the risk functions for $m=3$, $n=5$, $\sigma_1=20$ and $\sigma_3=0$.
The singular value shrinkage prior performs better than the Jeffreys prior, and the risk reduction increases as $\sigma_2$ decreases.
The Stein prior does not perform well because $\| M \|_{\rm F}$ is not small, even when $\sigma_2=0$.
Figure $\ref{sv_shr_fig}$ (b) shows the risk functions for $m=3$, $n=5$ and $\sigma_2=\sigma_3=0$.
Though the Stein prior performs best when $\sigma_1$ is small, 
its risk becomes almost the same as that of the Jeffreys prior as $\sigma_1$ increases.
On the other hand, the singular value shrinkage prior performs better than the Jeffreys prior regardless of the value of $\sigma_1$.
This is because the singular value shrinkage prior shrinks $\sigma_1, \sigma_2$ and $\sigma_3$ separately
and the Stein prior shrinks the Frobenius norm $\| M \|_{\rm F}^2 = \sigma_1^2+\sigma_2^2+ \sigma_3^2$.

Next, we compare the risk functions of the Bayesian predictive densities based on the singular value shrinkage prior to the Jeffreys prior and the Stein prior.
We sampled $(Y, \widetilde{Y})$ $10^4$ times and approximated the Kullback--Leibler risk by the sample mean of $\log \widetilde{p}(\widetilde{Y} \mid M) - \log \hat{p}_{\pi} (\widetilde{Y} \mid Y)$.

Figure $\ref{sv_shr_fig}$ (c) and (d) show the risk functions for $m=3$, $n=5$, $C=\widetilde{C}=I_n$, $\Sigma=\widetilde{\Sigma}=I_m$.
%(a) illustrates $\sigma_1=20$ and $\sigma_3=0$ and (b) shows $\sigma_2=\sigma_3=0$.
The performance of singular value shrinkage prior is qualitatively the same as in the estimation case.

Figure $\ref{sv_shr_fig}$ (e) and (f) show the risk functions for $m=10$, $n=20$, $\Sigma=\widetilde{\Sigma}=I_m$, $\sigma_1=20$, $\sigma_2=10$ and $\sigma_4=\cdots=\sigma_{10}=0$.
In (f), $C \neq \widetilde{C}$ and we used the singular value shrinkage prior depending on the future covariance $\eqref{prior_koba}$.
Singular value shrinkage priors work well for low rank matrices, even when $C \neq \widetilde{C}$.

%In practice, the covariance $\Sigma$ and $\widetilde{\Sigma}$ are usually unknown.
Finally, we consider the unknown variance case.
We assume $Y_i \sim {\rm N}_{n,m} (M, \sigma^2 Q, I_m) \ (i=1,\ldots,T)$ and $\widetilde{Y} \sim {\rm N}_{n,m} (M, \sigma^2 \widetilde{Q}, I_m)$,
%We assume $C = \sigma^2 Q$, $\Sigma = I_m$, $\widetilde{C} = \sigma^2 \widetilde{Q}$, and $\widetilde{\Sigma} = I_m$, 
where $T \geq 2$, $\sigma^2$ is unknown and $Q$ and $\widetilde{Q}$ are known.
This corresponds to multivariate linear regression with independent residuals.
%For vector-variate case $Y_i \sim {\rm N}_d (\mu, \sigma^2 {\rm I}_d) \ (i=1,\cdots,T)$ and $\widetilde{Y} \sim {\rm N}_d (\mu, \sigma^2 {\rm I}_d)$ where $d \geq 3$, 
%\cite{Kato} showed that the Bayesian predictive density based on the prior $\pi(\mu,\sigma) = \| \mu \|^{2-d} \sigma^{-1}$
%dominates that based on the right invariant prior $\pi(\mu,\sigma) = \sigma^{-1}$, which is minimax.
%Therefore, the Bayesian predictive density based on the prior $\pi(M,\sigma) = \pi_{{\rm SVS}} (M) \sigma^{-1}$ is expected to dominate that based on the right invariant prior $\pi(M,\sigma) = \sigma^{-1}$, which is minimax.
As a generalization of the result of \cite{Kato}, the Bayesian predictive density based on the prior $\pi(M,\sigma) = \pi_{{\rm SVS}} (M) \sigma^{-1}$ is expected to dominate that based on the right invariant prior $\pi(M,\sigma) = \sigma^{-1}$, which is minimax.
We verify this numerically. The predictive density is% represented as
\begin{equation*}
\hat{p}_{\pi} (\widetilde{Y} \mid Y_1, \ldots, Y_T) = \frac{m_{\pi} (Z, \widetilde{s}^2)}{m_{\pi} (\overline{Y}, s^2)} \frac{m_{\pi} (\overline{Y}, s^2)}{m_{\pi} (Y_1,\ldots, Y_T)} \frac{m_{\pi} (Y_1,\ldots, Y_T, Z)}{m_{\pi} (Z, \widetilde{s}^2)},
\end{equation*}
where 
\begin{align*}
\overline{Y} &=T^{-1} \sum_{i=1}^T Y_i, \quad Z=(T Q^{-1} + \widetilde{Q}^{-1})^{-1} (T Q^{-1} \overline{Y} + \widetilde{Q}^{-1} \widetilde{Y}), \\
s^2 &= \sum_{i=1}^T \| Q^{-1/2} (Y_i - \overline{Y}) \|_{{\rm F}}^2, \quad \widetilde{s}^2 = \sum_{i=1}^T \| Q^{-1/2} (Y_i - Z) \|_{{\rm F}}^2 + \| \widetilde{Q}^{-1/2} (\widetilde{Y} - Z) \|_{{\rm F}}^2,
\end{align*} 
and $m_{\pi}$ denotes the marginal distribution with prior $\pi$.
Here, $m_{\pi} (\bar{Y}, s^2)/m_{\pi} (Y_1,\ldots, Y_T)$ and $m_{\pi} (Y_1,\cdots, Y_T, Z)/m_{\pi} (Z, \widetilde{s}^2)$ do not depend on $\pi$ due to sufficiency.
%because $(\bar{Y},s^2)$ and $(Z,\widetilde{s}^2)$ are sufficient statistics for the mean matrix parameter $M$, 
%To calculate $m_{\pi} (Z, \widetilde{s}^2)/m_{\pi} (\bar{Y}, s^2)$, we use
Also,
\begin{align*}
m_{\pi} (\overline{Y}, s^2) &= \frac{1}{2 s^2} E \{ \pi(M) \}, \quad {\rm vec} (M) \sim t_{m n(T-1)} \left\{ {\rm vec} (\overline{Y}), I_m \otimes \frac{s^2}{mn(T-1) T} Q \right\}, \\
m_{\pi} (Z, \widetilde{s}^2) &= \frac{1}{2 \widetilde{s}^2} E \{ \pi(M) \}, \quad {\rm vec} (M) \sim t_{m n T} \left\{ {\rm vec} (Z), I_m \otimes \frac{\widetilde{s}^2}{m n T^2} \widetilde{Q} \right\},
\end{align*}
where $t_{\nu} (\mu, \Sigma)$ denotes the multivariate t-distribution with $\nu$ degrees of freedom, mean $\mu$ and covariance $\Sigma$.
Figure $\ref{sv_shr_fig}$ (g) and (h) show the risk functions for $T=4$, $m=3$, $n=5$ and $Q=\widetilde{Q}=I_n$.
Here, the variance $\sigma^2 = 4$ is unknown.
We confirmed that similar results are obtained even when $Q \neq \widetilde{Q}$.
%In (a), $Q=\widetilde{Q}=I_n$ and in (b), $Q=I_n$ and $\widetilde{Q}={\rm diag} (\widetilde{q}_i)$, $\widetilde{q}_1=\cdots=\widetilde{q}_{10}=1$, and $\widetilde{q}_{11}=\cdots=\widetilde{q}_{20}=2$.
The performance of singular value shrinkage prior is qualitatively the same as in the known variance case.
%Therefore, singular value shrinkage prior dominates the Jeffreys prior numerically even when the variance is unknown. 

\begin{figure}[htbp]
	\begin{minipage}{0.5\hsize}
		(a)
		\begin{center}
			\includegraphics[width=6cm]{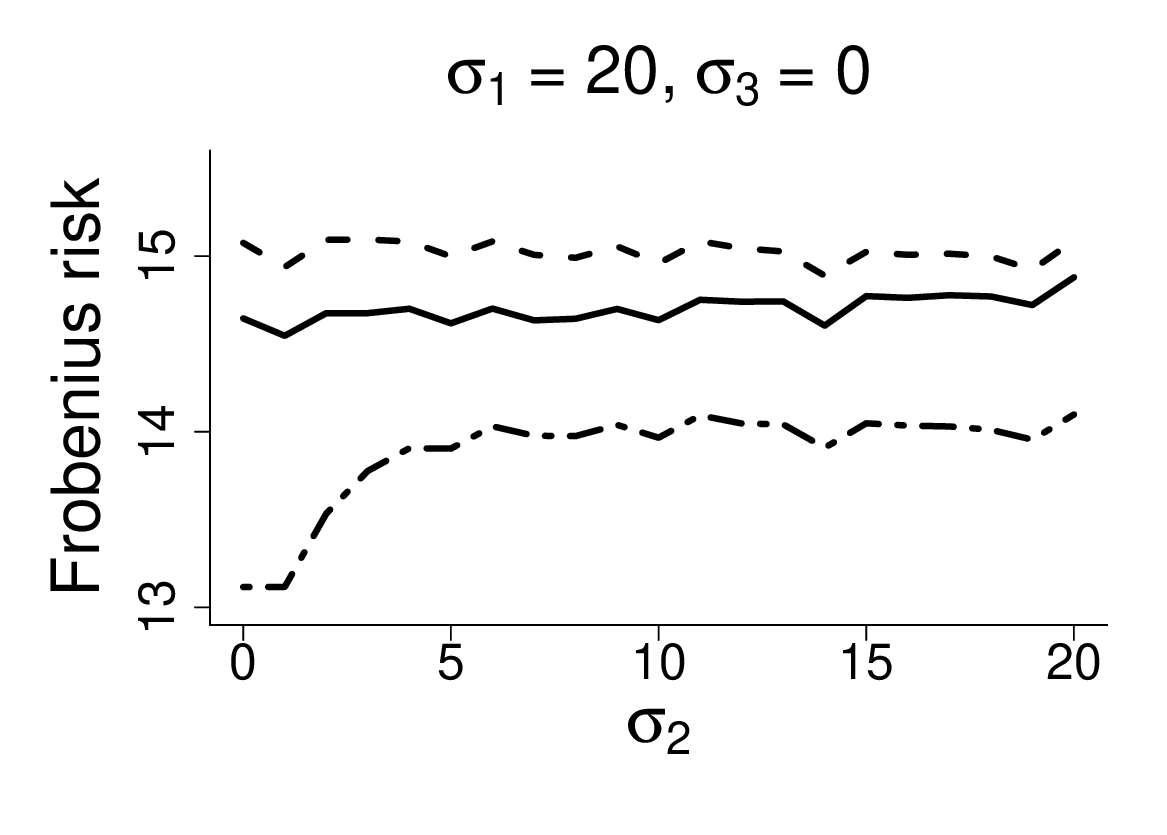}
		\end{center}
	\end{minipage}
	\begin{minipage}{0.5\hsize}
		(b)
		\begin{center}
			\includegraphics[width=6cm]{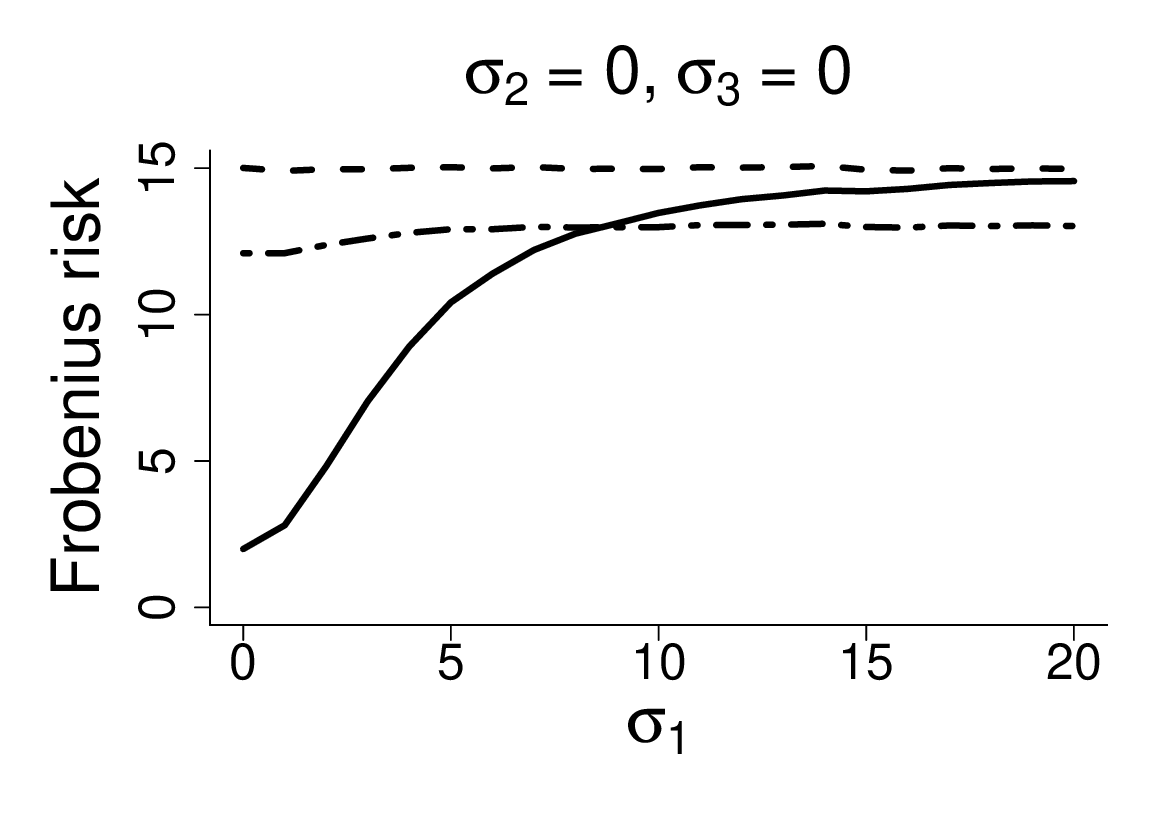}
		\end{center}
	\end{minipage}
	\begin{minipage}{0.5\hsize}
		(c)
		\begin{center}
			\includegraphics[width=6cm]{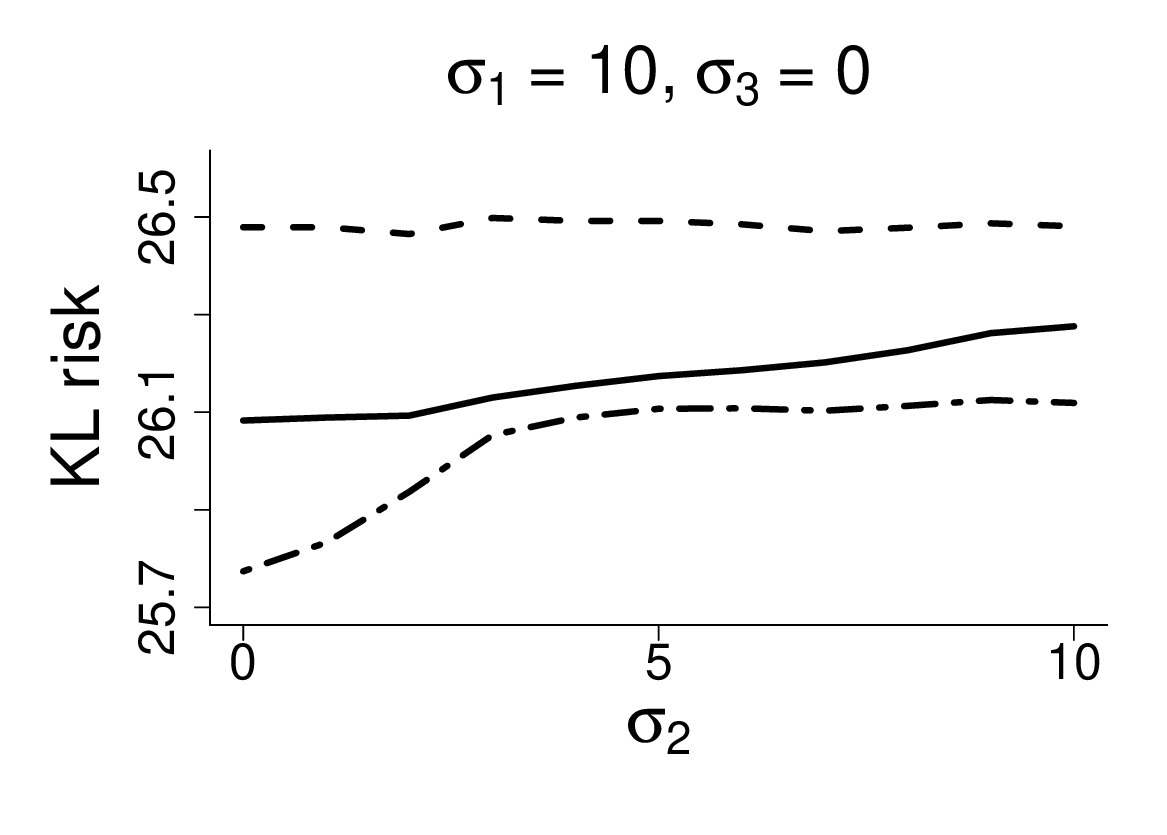}
		\end{center}
	\end{minipage}
	\begin{minipage}{0.5\hsize}
		(d)
		\begin{center}
			\includegraphics[width=6cm]{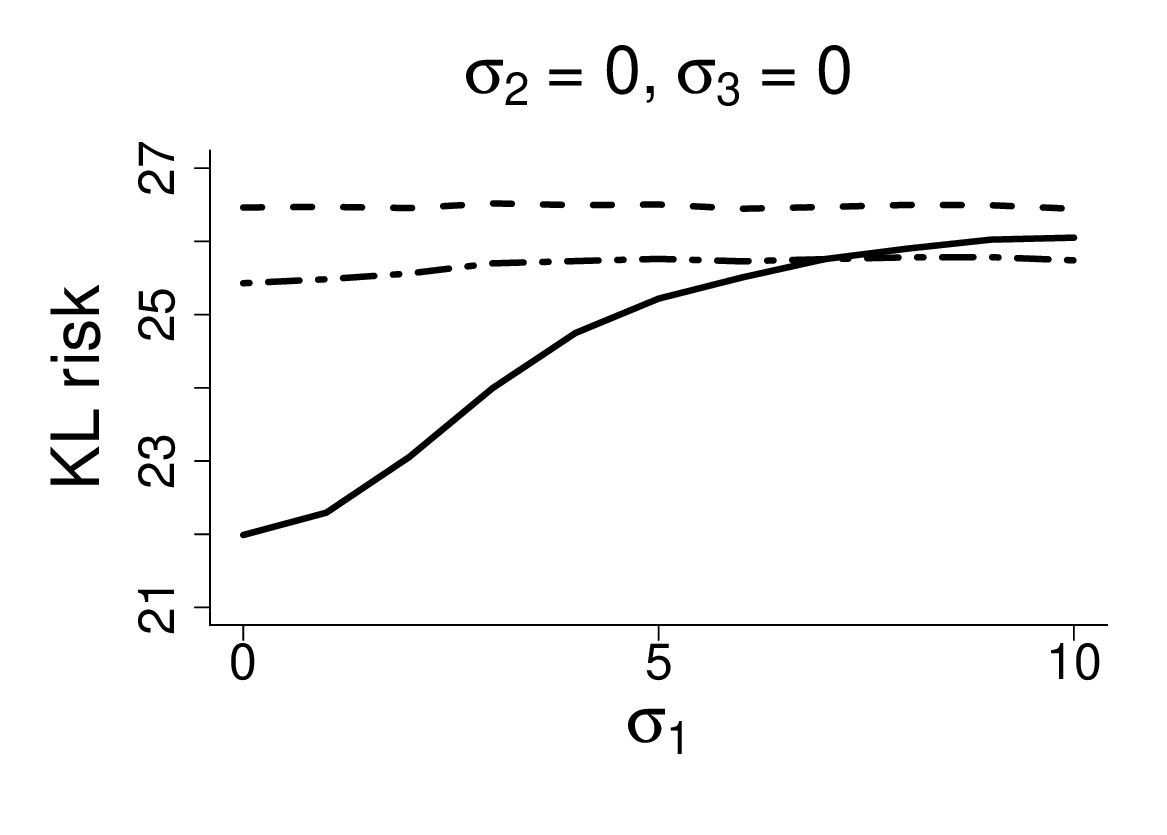}
		\end{center}
	\end{minipage}
	\begin{minipage}{0.5\hsize}
		(e)
		\begin{center}
			\includegraphics[width=6cm]{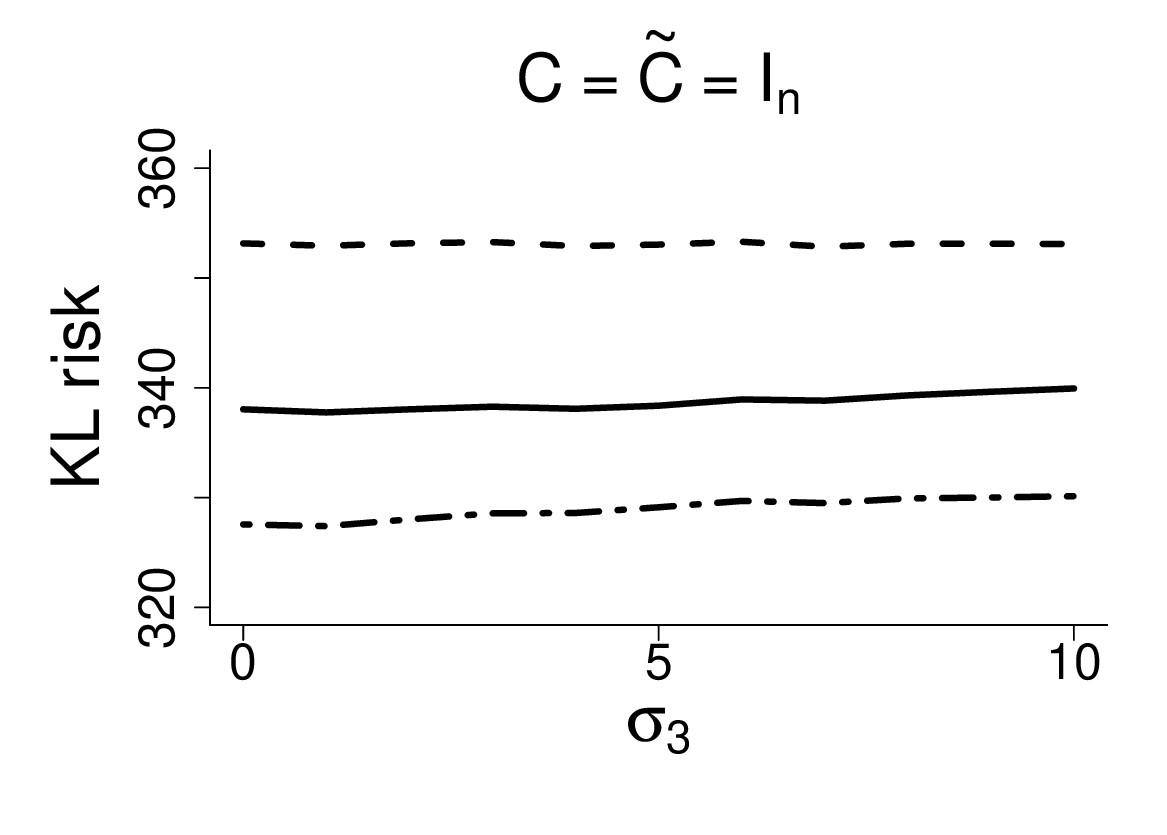}
		\end{center}
	\end{minipage}
	\begin{minipage}{0.5\hsize}
		(f)
		\begin{center}
			\includegraphics[width=6cm]{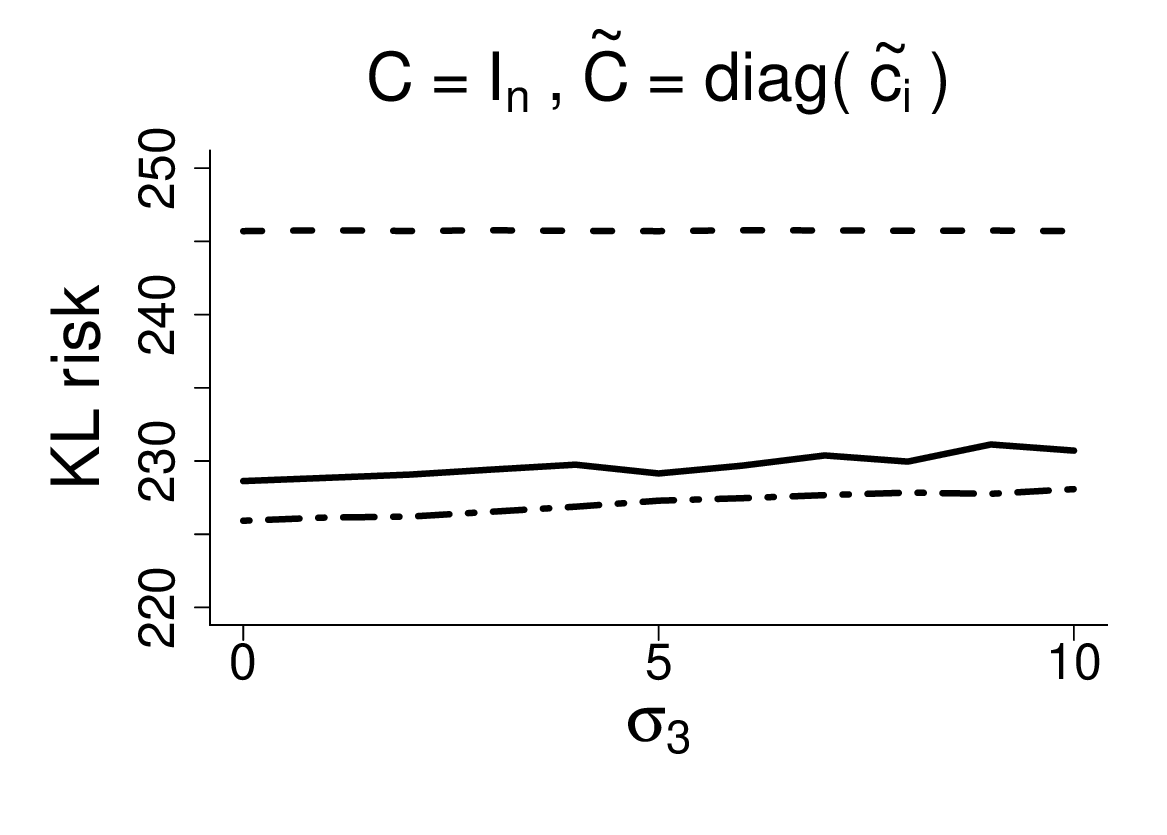}
		\end{center}
	\end{minipage}
	\begin{minipage}{0.5\hsize}
		(g)
		\begin{center}
			\includegraphics[width=6cm]{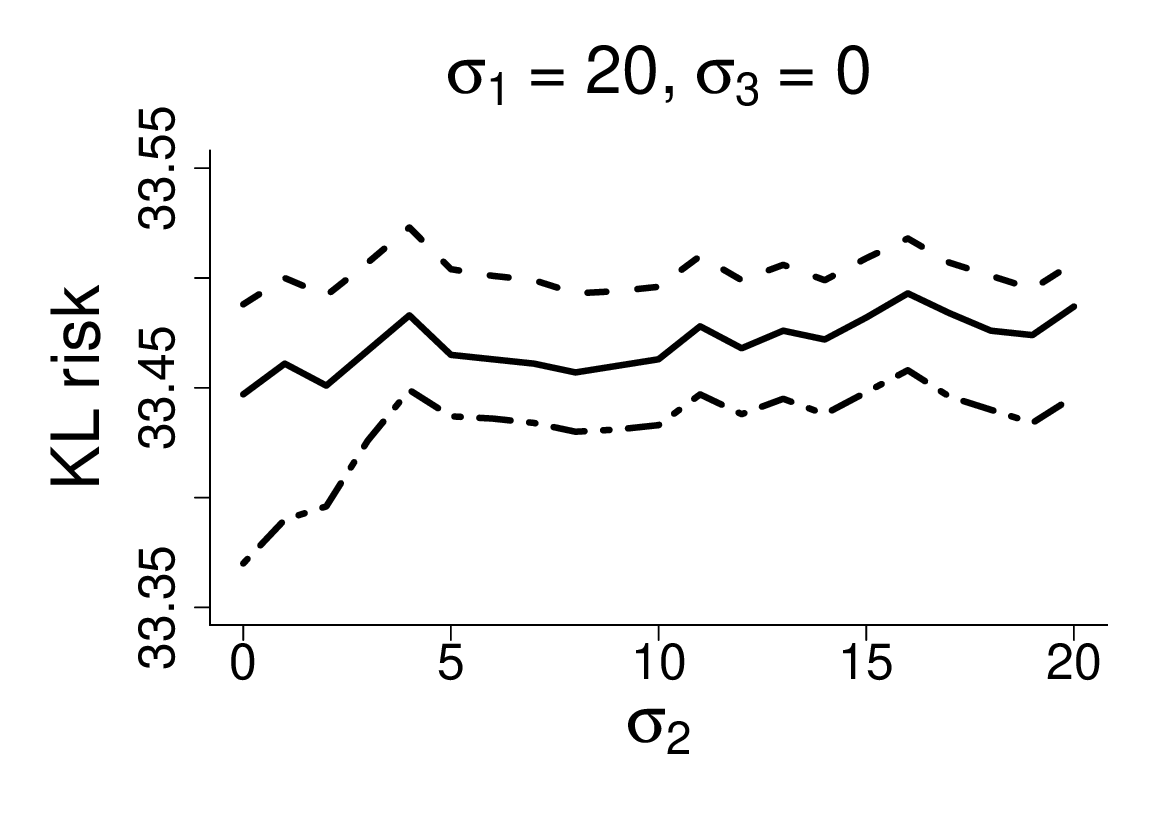}
		\end{center}
	\end{minipage}
	\begin{minipage}{0.5\hsize}
		(h)
		\begin{center}
			\includegraphics[width=6cm]{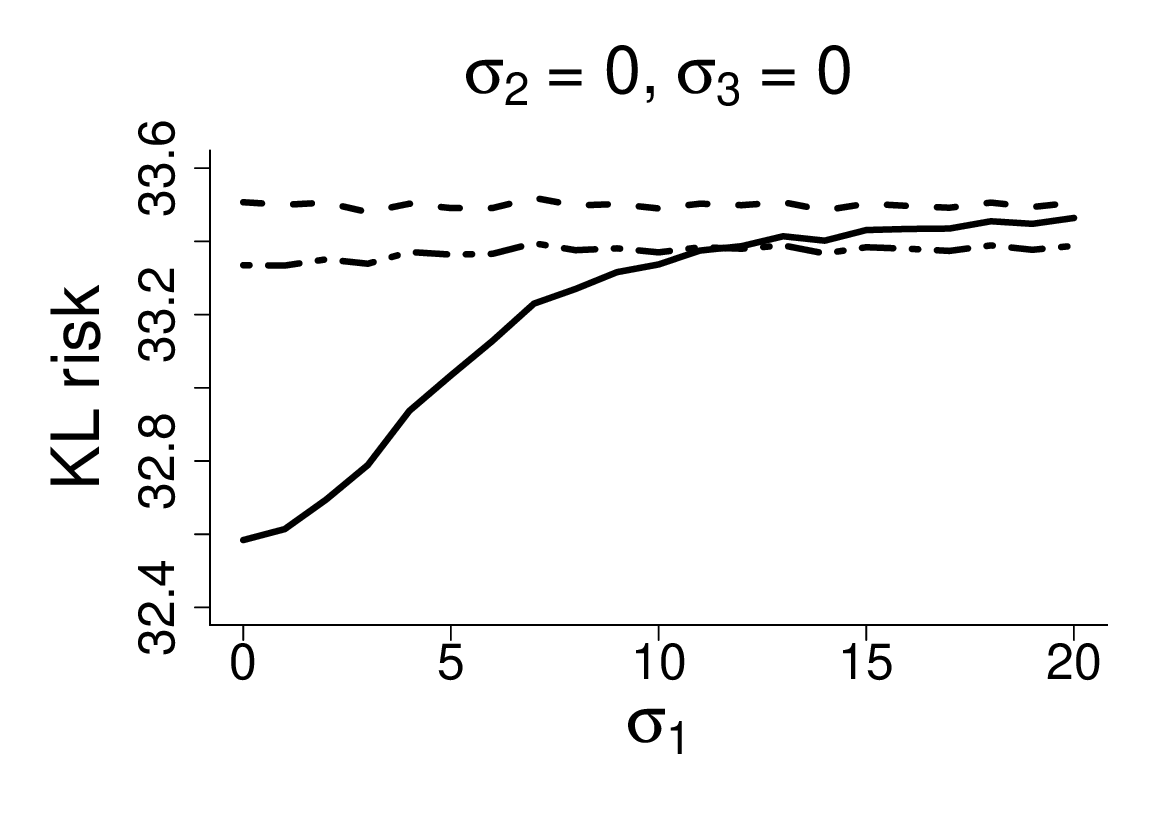}
		\end{center}
	\end{minipage}
	\caption{(a) (b) Risk functions of Bayes estimators when $m=3$, $n=5$, $C = I_n$ and $\Sigma = I_m$. (c) (d) Risk functions of Bayesian predictive densities when $m=3$, $n=5$, $C=\widetilde{C}=I_n$ and $\Sigma=\widetilde{\Sigma}=I_m$. (e) (f) Risk functions of Bayesian predictive densities when $m=10$, $n=20$, $\Sigma=\widetilde{\Sigma}=I_m$, $\sigma_1=20$, $\sigma_2=10$ and $\sigma_4=\cdots=\sigma_{10}=0$. In (f), $\widetilde{c}_1=\cdots=\widetilde{c}_{10}=1$ and $\widetilde{c}_{11}=\cdots=\widetilde{c}_{20}=2$. (g) (h) Risk functions of Bayesian predictive densities when $T=4$, $m=3$, $n=5$, $Q=\widetilde{Q}=I_n$ and $\sigma^2=4$ is unknown. dashed line: uniform prior, solid line: the Stein prior, dash-dot line: our prior}
	\label{sv_shr_fig}
\end{figure}

\section*{Acknowledgements}
The authors are grateful to the referees and the Associate Editor for valuable comments.
This work was supported by Grants-in-Aid for Scientific Research from the Japan Society for the Promotion of Science.

\section*{Appendix}
\subsection*{Proof of Theorem $\ref{th_piSVS_superharmonic}$}
First, we prove that $\det(M^{\top} M + \varepsilon I_{m})^{-(n-m-1)/2}$ is superharmonic for every $\varepsilon > 0$.

We write the $(i,j)$th entry of a matrix $X$ by $X_{i j}$ and the $(i,j)$th entry of $X^{-1}$ by $X^{i j}$.
Let $K = M^{\top} M + \varepsilon I$, so that $K_{a b} = \sum_i M_{i a} M_{i b} + \varepsilon \delta_{a b} = K_{b a}$,
where the subscripts $a$, $b$, $\ldots$ run from $1$ to $m$ and the subscripts $i$, $j$, $\ldots$ run from $1$ to $n$,
$\delta_{a b}$ is 1 when $a=b$ and 0 otherwise, and the $(i,j)$th entry of a matrix $X$ is denoted by $X_{ij}$.
From the definition,
\begin{equation}
\frac{\partial K_{b c}}{\partial M_{i a}} = \delta_{a c} M_{i b} + \delta_{a b} M_{i c}. \label{sigma_m}
\end{equation}

By using
\begin{equation*}
\frac{\partial}{\partial K_{ab}} \det K = K^{ab} \det K
\end{equation*}
and $\eqref{sigma_m}$, we obtain
\begin{align}
\frac{\partial}{\partial M_{i a}} \det K & =
\sum_{b,c} \frac{\partial K_{b c}}{\partial M_{i a}} \frac{\partial}{\partial K_{b c}} \det K
= 2 \sum_b M_{i b} K^{a b} \det K, \label{det_1st_diff}
\end{align}
where $K^{ab}$ is the $(a,b)$th entry of the inverse matrix of $K^{-1}$.
Therefore,
\begin{align}
\frac{\partial^2}{\partial M_{i a}^2} \det K
%= 2 \left( \sum_b \delta_{a b} K^{a b} \right) \det K + 2 \left( \sum_b M_{i b} \frac{\partial K^{a b}}{\partial M_{i a}} \right) \det K \nonumber \\
%& \quad \quad + 2 \left( \sum_b M_{i b} K^{a b} \right) \frac{\partial}{\partial M_{i a}} \det K \nonumber \\
=& 2 K^{a a} \det K - 2 \left( \sum_{b,d} M_{i b} M_{i d} K^{a a} K^{b d} \right) \det K \nonumber \\ 
&- 2 \left( \sum_{b,c} M_{i b} M_{i c} K^{a b} K^{a c} \right)
\det K + 4 \left( \sum_{b,c} M_{i b} M_{i c} K^{a b} K^{a c} \right) \det K. \label{det_2nd_diff}
\end{align}
Here, we used
\begin{equation*}
\frac{\partial K^{a b}}{\partial M_{i a}} = -\sum_c M_{i c} K^{a a} K^{b c} - \sum_d M_{i d} K^{a d} K^{a b},
\end{equation*}
which can be derived by differentiating the equation $\sum_c K_{b c} K^{a c} = \delta_{a b}$.

We have
\begin{align}
\Delta (\det K)^{-(n-m-1)/2} & = \sum_{i,a} \frac{\partial^2}{\partial M_{i a}^2} (\det K)^{-(n-m-1)/2} \nonumber \\
& = \frac{n-m-1}{2} (\det K)^{-(n-m-1)/2} \sum_{i, a} (A_{i a} + B_{i a}), \label{Delta_K}
\end{align}
where
\begin{equation*}
A_{i a} = \frac{n-m+1}{2} (\det K)^{-2} \left( \frac{\partial}{\partial M_{i a}} \det K \right)^2, \quad
B_{i a} = - (\det K)^{-1} \frac{\partial^2}{\partial M_{i a}^2} \det K.
\end{equation*}
%
%Therefore, we calculate the sum of $A_{i a}$ and $B_{i a}$.

By using $\eqref{det_1st_diff}$, we obtain $A_{i a} = 2 (n-m+1) \left( \sum_b M_{i b} K^{a b} \right) \left( \sum_c M_{i c} K^{a c} \right)$.
Thus,
\begin{align*}
\sum_i A_{i a} %& = 2 (n-m+1) \sum_{b,c} \left( K^{a b} K^{a c} \sum_i M_{i b} M_{i c} \right) \nonumber \\
%	& = 2 (n-m+1) \sum_{b,c} K^{a b} K^{a c} (K_{b c} - \varepsilon \delta_{b c}) \notag \\
&= 2 (n-m+1) \left( K^{a a} - \varepsilon \sum_b K^{a b} K^{a b} \right).
\end{align*}
On the other hand, by $\eqref{det_2nd_diff}$,
\begin{equation*}
\sum_i B_{i a} = -2 (n-m+1) K^{a a} -2 \varepsilon K^{a a} \sum_b K^{b b} + 2 \varepsilon \sum_b K^{a b} K^{a b}.
\end{equation*}
Hence, noting that $K^{b b} = \sum_i M_{i b}^2 + \varepsilon > 0$, we obtain
\begin{equation}
\sum_i (A_{i a} + B_{i a}) = -2(n-m) \varepsilon \sum_b K^{a b} K^{a b} -2 \varepsilon K^{a a} \sum_b K^{b b} < 0. \label{A+B}
\end{equation}

Thus, from $\eqref{Delta_K}$ and $\eqref{A+B}$,
\begin{equation}
\Delta (\det K)^{-(n-m-1)/2} < 0. \label{superharmonic_K}
\end{equation}
Therefore, $\det (M^{\top} M + \varepsilon I)^{-(n-m-1)/2}$ is superharmonic for every $\varepsilon > 0$.

Now, let $\pi^{(k)} (M) = \det \left( M^{\top} M + k^{-1} I_m \right)^{-(n-m-1)/2}$.
Then, $\pi^{(k)}$ is superharmonic by $\eqref{superharmonic_K}$ and $\pi^{(1)} \leq \pi^{(2)} \leq \cdots$, since
\begin{equation*}
\pi^{(k)} (M) = \prod_{i=1}^m  \left\{ \lambda_i (M^{\top} M) + \frac{1}{k} \right\}^{-(n-m-1)/2},
\end{equation*}
where $\lambda_i (M^{\top} M) \geq 0$ denotes the $i$th eigenvalue of $M^{\top} M$.
Also, $\lim_{k \rightarrow \infty} \pi^{(k)}(M) = \pi_{{\rm SVS}} (M)$ for every $M$.
Therefore, by Theorem 3.4.8 of \cite{Helms}, $\pi_{{\rm SVS}}$ is superharmonic.

\subsection*{Proof of Lemma $\ref{th_m_superharmonic}$}
Let $\phi (Y; C,\Sigma) = p(Y \mid O,C,\Sigma)$.
Then, putting $A = Y-M$,
\begin{align*}
m_{\pi} (Y; C,\Sigma) &= \int \phi (Y-M; C,\Sigma) \pi (M) {\rm d} M
= \int \phi (A; C,\Sigma) \pi (Y-A) {\rm d} A.
\end{align*}
Let $d=mn$. Now, for every $x \in \mathbb{R}^d$ and $\delta > 0$,
\begin{align*}
\frac{1}{\Omega_d \delta^{d-1}} \int_{S_{x,\delta}} m_{\pi}(Y; C,\Sigma) {\rm d} s(Y)
=& \frac{1}{\Omega_d \delta^{d-1}} \int_{S_{x,\delta}} \left\{ \int \phi (A; C,\Sigma) \pi (Y-A) {\rm d} A \right\} {\rm d} s(Y) \nonumber \\
=& \frac{1}{\Omega_d \delta^{d-1}} \int \phi (A; C,\Sigma) \left\{ \int_{S_{x,\delta}} \pi (Y-A) {\rm d} s(Y) \right\} {\rm d} A \nonumber \\
\leq & \frac{1}{\Omega_d \delta^{d-1}} \int \phi (A; C,\Sigma) \pi (-A) {\rm d} A
= m_{\pi} (O).
\end{align*}
Here, the second equation follows from Fubini's theorem and the inequality follows from the superharmonicity of $\pi$.
Therefore, $m_{\pi} (Y; C,\Sigma)$ is superharmonic.

\subsection*{Proof of Proposition $\ref{th_george}$}
From Lemma 2 of \cite{George06}, the difference of $R_{{\rm KL}}$ is
%\begin{align*}
%	R_{{\rm KL}} & (M, \hat{p}_{{\rm I}}) - R_{{\rm KL}} (M, \hat{p}_{\pi}) \nonumber \\
%	= & {\rm E}_{M, v_0} \log m_{\pi} (Z; v_0 I_n, I_m) - {\rm E}_{M, v_1} \log m_{\pi} (Z; v_1 I_n, I_m),
%\end{align*}
\begin{equation*}
R_{{\rm KL}} (M, \hat{p}_{{\rm I}}) - R_{{\rm KL}} (M, \hat{p}_{\pi}) = E_{M, v_0} \left\{ \log m_{\pi} (Z; v_0 I_n, I_m) \right\} - E_{M, v_1} \left\{ \log m_{\pi} (Z; v_1 I_n, I_m) \right\}, \label{KLrisk_diff0}
\end{equation*}
where $E_{M,v} (\cdot)$ denotes the expectation with respect to $Z \sim {\rm N}_{n,m} (M, v I_n, I_m)$. Then,
\begin{align}
R_{{\rm KL}} (M, \hat{p}_{{\rm I}}) - R_{{\rm KL}} (M, \hat{p}_{\pi}) = & E_{M, v_0} \left\{ \log m_{\pi} (Z; v_0 I_n, I_m) - \log m_{\pi} (Z; v_1 I_n, I_m) \right\} \nonumber \\
& + \left\{ E_{M, v_0} \log m_{\pi} (Z; v_1 I_n, I_m) - E_{M, v_1} \log m_{\pi} (Z; v_1 I_n, I_m) \right\}. \label{KLrisk_diff}
\end{align}
%Since $\pi$ is superharmonic, from Lemma 3.4.4 of \cite{Helms} and $v_0 < v_1$, 
Since $\pi$ is superharmonic, from Lemma $\ref{lemma_helms}$ and $v_0 < v_1$, 
\begin{align*}
m_{\pi} (Z; v_0 I_n, I_m) &= E_{M \sim {\rm N}_{n,m} (Z, v_0 I_n, I_m)} \pi(M) \nonumber \\
& \geq E_{M \sim {\rm N}_{n,m} (Z, v_1 I_n, I_m)} \pi(M)
= m_{\pi} (Z; v_1 I_n, I_m).
\end{align*}
Thus, the first term of the right hand side of $\eqref{KLrisk_diff}$ is nonnegative.
Also, since $m_{\pi}$ is superharmonic from Lemma $\ref{th_m_superharmonic}$ and the logarithm of a superharmonic function is superharmonic, 
$\log m_{\pi} (Z; v_1 I_n, I_m)$ is superharmonic.
%Then, from Lemma 3.4.4 of \cite{Helms} and $v_0 < v_1$, the second term of the right hand side of $\eqref{KLrisk_diff}$ is nonnegative.
Then, from Lemma $\ref{lemma_helms}$ and $v_0 < v_1$, the second term of the right hand side of $\eqref{KLrisk_diff}$ is nonnegative.
Therefore, $\eqref{KLrisk_diff}$ is nonnegative.

\subsection*{Proof of Theorem $\ref{th_predictive_density}$}
We represent the Bayesian predictive density as
\begin{equation}
\hat{p}_{\pi} (\widetilde{Y} \mid Y) = \frac{m_{\pi} (Y, \widetilde{Y})}{m_{\pi} (Y)} = \frac{m_{\pi} (Y, \widetilde{Y})}{m_{\pi} (Z)} \frac{m_{\pi} (Z)}{m_{\pi} (Y)}, \label{p_pi}
\end{equation}
where $m_{\pi}$ denotes the marginal distribution with prior $\pi$.
Here, $m_{\pi} (Y,\widetilde{Y})/m_{\pi} (Z)$ does not depend on $\pi$, because $Z$ is sufficient for $M$.
Therefore, we obtain
\begin{equation}
\frac{m_{\pi} (Y, \widetilde{Y})}{m_{\pi} (Z)} = \frac{p (Y, \widetilde{Y} \mid M)}{p (Z \mid M)}
= \pi^{-m n/2} {\rm etr} \left\{ -\frac{1}{4} (\widetilde{Y}-Y)^{\top} (\widetilde{Y}-Y) \right\}. \label{m_pi_ratio}
\end{equation}
Next, we calculate $m_{{\rm SVS}} (Y)$ and $m_{{\rm SVS}} (Z)$. From the definition,
\begin{equation*}
m_{{\rm SVS}} (Y) = \int p(Y \mid M) \pi(M) {\rm d} M.
\end{equation*}
Now we interpret $p(Y \mid M)$ as the probability density of $M \sim {\rm N}_{n,m} (Y, I, I)$.
Then $m_{{\rm SVS}} (Y)$ is viewed as the expectation of $\pi(M)$.
From Theorem 3.5.1 in \cite{Gupta}, $S = M^{\top} M$ has a noncentral Wishart distribution $S \sim W_m (n, I_m, Y^{\top} Y)$.
Therefore,
\begin{equation*}
m_{{\rm SVS}} (Y) = E \left\{ (\det S)^{-(n-m-1)/2} \right\}.
\end{equation*}
By using Theorem 3.5.6 in \cite{Gupta}, we obtain
\begin{align*}
E &\left[ (\det S)^{-(n-m-1)/2} \right]
&= \frac{2^{-\frac{m(n-m-1)}{2}} \Gamma_m \left( \frac{m+1}{2} \right)}
{\Gamma_m \left( \frac{n}{2} \right)} {\rm etr} \left( -\frac{1}{2} Y^{\top} Y \right)
{}_1 F_1 \left( \frac{m+1}{2}; \frac{n}{2}; \frac{1}{2} Y^{\top} Y \right).
\end{align*}
Using this, we obtain
\begin{align}
m_{{\rm SVS}} (Y) & = \frac{2^{-\frac{m(n-m-1)}{2}} \Gamma_m \left( \frac{m+1}{2} \right)}{\Gamma_m \left( \frac{n}{2} \right)} {\rm etr} \left( -\frac{1}{2} Y^{\top} Y \right) {}_1 F_1 \left( \frac{m+1}{2}; \frac{n}{2}; \frac{1}{2} Y^{\top} Y \right). \label{marginal}
\end{align}
Similarly, we obtain
\begin{equation}
m_{{\rm SVS}} (Z) = \frac{2^{-m n} \Gamma_m \left( \frac{m+1}{2} \right)}{\Gamma_m \left( \frac{n}{2} \right)} {\rm etr} \left( -\frac{1}{4} Z^{\top} Z \right) {}_1 F_1 \left( \frac{m+1}{2}; \frac{n}{2}; \frac{1}{4} Z^{\top} Z \right). \label{m_SVS_z}
\end{equation}

Substituting $\eqref{m_pi_ratio}$, $\eqref{marginal}$ and $\eqref{m_SVS_z}$ into $\eqref{p_pi}$, we obtain the result.


\begin{thebibliography}{99}
	\expandafter\ifx\csname natexlab\endcsname\relax\def\natexlab#1{#1}\fi
	
	\bibitem[{Dawid(1981)}]{Dawid}
	\textsc{Dawid, A. P.} (1981).
	\newblock{Some matrix-variate distribution theory: notational considerations and a Bayesian application}.
	\newblock \textit{Biometrika} \textbf{68}, 265--74.
	
	\bibitem[{Efron \& Morris(1972)}]{Efron}
	\textsc{Efron, B.} \& \textsc{Morris, C. N.} (1972).
	\newblock{Empirical Bayes on vector observations: an extension of Stein's method}.
	\newblock \textit{Biometrika} \textbf{59}, 335--47.
	
	\bibitem[{George et al.(2006)}]{George06}
	\textsc{George, E. I.}, \textsc{Liang, F.} \& \textsc{Xu, X.} (2006).
	\newblock{Improved minimax predictive densities under Kullback--Leibler loss}.
	\newblock \textit{Ann. Statist.} \textbf{34}, 78--91.
	
	\bibitem[{George \& Xu(2008)}]{George08}
	\textsc{George, E. I.} \& \textsc{Xu, X.} (2008).
	\newblock{Predictive density estimation for multiple regression}.
	\newblock \textit{Economet. Theor.} \textbf{24}, 528--44.
	
	\bibitem[{Gupta \& Nagar(2000)}]{Gupta}
	\textsc{Gupta, A. K.} \& \textsc{Nagar, D. K.} (2000).
	\newblock \textit{Matrix Variate Distributions}.
	\newblock New York: Chapman \& Hall.
	
	\bibitem[{Helms(2009)}]{Helms}
	\textsc{Helms, L. L. } (2009).
	\newblock \textit{Potential Theory}.
	\newblock New York: Springer-Verlag.
	
	\bibitem[{Kato(2009)}]{Kato}
	\textsc{Kato, K.} (2009).
	\newblock{Improved prediction for a multivariate normal distribution with unknown mean and variance}.
	\newblock \textit{Ann. I. Stat. Math.} \textbf{61}, 531--42.
	
	\bibitem[{Kobayashi \& Komaki(2008)}]{Kobayashi}
	\textsc{Kobayashi, K.} \& \textsc{Komaki, F.} (2008).
	\newblock{Bayesian shrinkage prediction for the regression problem}.
	\newblock \textit{J. Multivariate. Anal.} \textbf{99}, 1888--1905.
	
	\bibitem[{Komaki(2001)}]{Komaki01}
	\textsc{Komaki, F.} (2001).
	\newblock{A shrinkage predictive distribution for multivariate normal observables}.
	\newblock \textit{Biometrika} \textbf{88}, 859--64.
	
	\bibitem[{Reinsel \& Velu(1998)}]{Reinsel}
	\textsc{Reinsel, G. C.} \& \textsc{Velu, R. P.} (1998).
	\newblock \textit{Multivariate Reduced-Rank Regression}.
	\newblock New York: Springer-Verlag.
	
	\bibitem[{Stein(1974)}]{Stein74}
	\textsc{Stein, C.} (1974).
	\newblock{Estimation of the mean of a multivariate normal distribution}.
	\newblock In \textit{Proceedings of the Prague Symposium on Asymptotic Statistics}, Ed. J. Hajek, pp. 345-81. Prague: Universita Karlova.
	
	\bibitem[{Tsukuma(2008)}]{Tsukuma08}
	\textsc{Tsukuma, H.} (2008).
	\newblock{Admissibility and minimaxity of Bayes estimators for a normal mean matrix}.
	\newblock \textit{J. Multivariate. Anal.} \textbf{99}, 2251--64.
	
\end{thebibliography}
\end{document}